%% file: main.tex
\documentclass[a4paper,12pt]{amsart}
\usepackage{amsmath,amssymb,amsthm}    
\usepackage{comment}
\usepackage{ulem}
\usepackage{tikz}
\usepackage{mathdots}
\usepackage{tikz}
\usepackage{float}
\usepackage{subcaption}
\usepackage[subrefformat=parens,labelformat=parens]{subcaption}
\usetikzlibrary{calc,positioning,decorations.pathmorphing,decorations.pathreplacing}
\tikzset{emp/.style={double distance = 0.3ex}}

\tikzset{M edge/.style={line width=1.3pt,double distance=1.1pt}}
\tikzset{F1 edge/.style={line width=1.3,color=red,->}}
\tikzset{F2 edge/.style={line width=1.3,color=blue,->}}
\tikzset{E edge/.style={line width=1.3,color=black,-}}

\tikzset{squared black vertex/.style={draw,minimum size=2mm,inner sep=0pt,outer sep=3pt,fill=black, color=black}}
\tikzset{red vertex/.style={circle,draw,minimum size=2mm,inner sep=0pt,outer sep=2pt,fill=red, color=red}}
\tikzset{blue vertex/.style={circle,draw,minimum size=2mm,inner sep=0pt,outer sep=2pt,fill=blue, color=blue}}
\tikzset{black vertex/.style={circle,draw,minimum size=2mm,inner sep=0pt,outer sep=2pt,fill=black, color=black}}
\tikzset{small black vertex/.style={circle,draw,minimum size=1.2mm,inner sep=0pt,outer sep=1.2pt,fill=black, color=black}}
\tikzset{small white vertex/.style={circle,draw,minimum size=1.2mm,inner sep=0pt,outer sep=1.2pt,color=black,fill=white}}
\tikzset{square vertex/.style={draw,minimum size=1.2mm,inner sep=0pt,outer sep=1.2pt,fill=black, color=red}}
\tikzset{white vertex/.style={circle,draw,minimum size=2mm,inner sep=0pt,outer sep=3pt,color=black,fill=white}}

\tikzset{fatpath/.style={line width=9pt,rounded corners=.1mm}}

\tikzstyle{edge}=[line width=1.3]

\tikzstyle{color1}=[color=blue] 
\tikzstyle{color2}=[color=red]
\tikzstyle{color3}=[color=green] 
\tikzstyle{color4}=[fill=yellow]
\tikzstyle{color5}=[ dashed] 

\tikzstyle{backcolor1}=[color=gray!55!white] 
\tikzstyle{backcolor2}=[color=blue!35!white] 

\usepackage{standalone}
\usepackage{mathtools}

\DeclareUnicodeCharacter{2212}{\ensuremath{-}}

\usepackage{enumitem}
\newcommand{\etal}{\textit{et~al.}}
\usepackage{amsaddr}

\usepackage[T1]{fontenc}
\usepackage[english]{babel}

\usepackage{fullpage}

\newtheorem{theorem}             {Theorem}
\newtheorem{lemma}     	[theorem] {Lemma}        
\newtheorem{conjecture}	[theorem] {Conjecture}

\newtheorem{proposition}[theorem] {Proposition}   
\newtheorem{corollary}	[theorem] {Corollary}
     
\newtheorem{claim}	[theorem] {Claim}

\usepackage{diagbox}
\newtheoremstyle{case}{}{}{}{}{\bfseries}{:}{ }{}
\theoremstyle{case}
\newtheorem{case}{Case}

\title{On Tuza's conjecture in dense graphs} 
\thanks{
  J. Gutiérrez and L. Chahua were partially supported by Fondo Semilla UTEC 871075-2022.
  }

\author{Luis Chahua}
\address{\vspace{-4mm}Departamento de Ciencia de la Computación\\ 
Universidad de Ingeniería y Tecnología (UTEC), Perú}
\email{luis.chahua@utec.edu.pe}

\author{Juan Gutiérrez} 
\address{\vspace{-4mm}Departamento de Ciencia de la Computación\\ 
Universidad de Ingeniería y Tecnología (UTEC), Perú}
\email{jgutierreza@utec.edu.pe}

\usepackage{lineno}

\begin{document}
\normalem

\maketitle

\begin{abstract}
In 1982, Tuza conjectured that the size~\(\tau(G)\) of a minimum set of
	edges that intersects every triangle of a graph~\(G\) is at most twice the size~\(\nu(G)\) 
	of a maximum set of edge-disjoint triangles of~\(G\).
	This conjecture was proved for several graph classes.
	In this paper, we present three results regarding Tuza's Conjecture for dense graphs. 
By using a probabilistic argument, Tuza proved its conjecture for graphs on $n$ vertices with minimum degree at least~$\frac{7n}{8}$.
We extend this technique to show that
Tuza's conjecture is valid for split graphs with minimum degree at least $\frac{3n}{5}$; and that~\(\tau(G) < \frac{28}{15}\nu(G)\) for every tripartite graph with minimum degree more than $\frac{33n}{56}$.
Finally, we show that~\({\tau(G)\leq \frac{3}{2}\nu(G)}\) when~$G$ is a complete 4-partite graph.
	Moreover, this bound is tight.
\end{abstract}

\maketitle
\section{Introduction and Preliminaries}\label{sec:intro}

\input{intro.tex}

\section{Dense split graphs}

\input{split2.tex}

\section{Dense tripartite graphs}

\input{3partite.tex}

\input{4partite.tex}

\section{Concluding remarks} \label{sec:remarks}

\input{remarks.tex}

\bibliographystyle{plain}
\bibliography{bibliografia}

\end{document}

%% file: intro.tex
In this paper, all graphs considered are simple and the notation and terminology are standard \cite{BondyM08,Diestel10}.
A \textit{triangle hitting}
of a graph~$G$ is a set of edges of~$G$ whose removal results in a triangle-free graph; and a \textit{triangle packing} of~$G$ is a set of pairwise
edge-disjoint triangles of~$G$. We denote by~$\tau(G)$ (resp.~$\nu(G)$) the cardinality of a minimum triangle hitting (resp. maximum
triangle packing) of~$G$. In 1981, Tuza posed the following conjecture.

\begin{conjecture}[\cite{Tuza81}]\label{conjecture:tuza}
  For every graph~\(G\), we have~\(\tau(G)\leq 2\,\nu(G)\).
\end{conjecture}%

Haxell \etal~\cite{Haxell99} showed the first and unique nontrivial bound to Tuza’s Conjecture. She showed
that~$\tau(G) \leq 2.87 \nu(G)$ for every graph~$G$.
Tuza showed his conjecture for planar graphs \cite{Tuza90}. Cui \etal~\cite{Cui09} characterized planar graphs for which Tuza’s Conjecture is tight. Haxell \etal~\cite{Haxell12} showed that, when~$G$ is a~$K_4$-free planar graph, the stronger inequality~$\tau(G) \leq \frac{3}{2} \nu(G)$ holds.
Botler \etal~\cite{Botler2020} showed the same bound for planar triangulations.
The purpose of this paper is to study Tuza's conjecture for dense graphs.
In this direction, by using a probabilistic argument, Tuza proved his conjecture for graphs on~$n$ vertices and at least~$\frac{7}{16}n^2$ edges \cite{Tuza90}. By extending this technique, we show new results for the classes of split graphs 
and tripartite graphs. Finally, we show a tight result for complete 4-partite graphs. We next state what our results are and how they are related to the current literature.

 Botler \etal~proved Tuza's conjecture for~$K_8$-free chordal graphs \cite[Corollary 3.6]{Botler2020}. But the conjecture is still open for several other important subclasses of chordal graphs, as split graphs. In this direction, Bonamy~\etal~verified this conjecture for threshold graphs~\cite{Bonamy2022}, that is, graphs that are both split and cographs.

Our first result is the following.

\begin{itemize}
\item For every split graph on~$n$ vertices with minimum degree at least~$\frac{3n}{5}$, Tuza's conjecture holds.
\end{itemize}

For tripartite graphs, Haxell and Kohayakawa \cite{Haxell98} showed that~$\tau(G) \leq 1.956\nu(G)$. This bound was improved by Szestopalow {\cite[Theorem 4.1.5]{Sze16}}. He showed that~$\tau(G) \leq 1.87\nu(G)$.
Regarding tripartite graphs, we prove the following.

\begin{itemize}
\item For every tripartite graph~$G$ on~$n$ vertices and~$m>\frac{n^2}{4}$ edges,~${\tau(G) \leq \frac{n^2}{3(4m-n^2)} \cdot \nu(G)}$, which implies that~$\tau(G) < 1.8 \nu(G)$ if~$G$ has minimum degree at least~$0.59n$. 
\end{itemize}

Aparna \etal~\cite[Corollary 7]{Aparna11} showed that Tuza's Conjecture holds for
4-partite graphs.
We improve this result for complete 4-partite graphs by showing the following tight upper bound.
\begin{itemize}
\item 
For every complete 4-partite graph~$G$ on at least 5 vertices,
$\tau(G) \leq \frac{3}{2} \nu(G)$.
\end{itemize}

We begin by introducing some notation.
For a graph~$G$ and~$X \subseteq V(G)$, we denote by~$G[X]$ the subgraph induced by~$X$. Also, for $v \in V(G)$, $N(v) := \{u : uv \in E(G)\}$ is called the \textit{neighborhood} of $v$ and $N[v] := N(v) \cup \{v\}$ is its \textit{closed neighborhood}. 
We denote the complete graph on~$n$ vertices by~$K_{n}$. Also, we denote by~$d_{G}(v)$ the degree of~$v$ in a given graph~$G$. If the context is clear, we write~$d(v)$. The maximum degree of a graph $G$
is denoted by $\Delta(G)$ and the minimum degree by $\delta(G)$.

A graph~$G$ is called a~$(k,\ell)$-$\textit{graph}$ if~$V(G)$ has a partition~$\{X_1,X_2,\ldots,X_{k+\ell}\}$ such that~$X_i$ is a clique for~$1\leq i \leq k$ and it is an independent set otherwise. In that case,~$G$ is denoted by~$(X_1,X_2,\ldots,X_{k+\ell}, E(G))$. A~$(k,\ell)$-graph~$G = (X_1,X_2,\ldots,X_{k+\ell}, E(G))$ is called~$\textit{complete}$ if, for every~${1\leq i < j\leq k+\ell}$, any vertex in~$X_i$ is adjacent to any vertex in~$X_j$. A~$(1,1)$-graph~$G$ is called a~$\textit{split graph}$, a~$(0,2)$-graph~$G$ is called a~$\textit{bipartite graph}$, a~$(0,3)$-graph~$G$ is called a~$\textit{tripartite graph}$ and a~$(0,4)$-graph~$G$ is called a~$\textit{4-partite graph}$.

For a graph~$G$, we denote by~$\mathfrak{S}(G)$ the set of all permutations of~$V(G)$ and by~$\mathcal{T}(G)$ the set of all triangles in~$G$. A~$\textit{matching}$ in a graph~$G$ is a set of edges that are pairwise not incident. Let $G = (K,S,E(G))$ be a split graph. We say that $G$ is a \textit{threshold graph} if there exists a permutation of vertices $x_1x_2\cdots x_{|K|} \in \mathfrak{S}(K)$ and $y_1y_2\cdots y_{|S|} \in \mathfrak{S}(S)$ such that $N[x_{i+1}] \subseteq N[x_i]$ for all $1 \leq i < |K|$ and $N(y_i) \subseteq N(y_{i+1})$ for all $1 \leq i < |S|$.  

An~$\textit{edge-coloring}$ of a graph~$G$ is a collection~$\{M_1,M_2,\ldots,M_k\}$ of pairwise disjoint matchings in~$G$ such that~$M_1 \cup M_2 \cup \cdots \cup M_k = E(G)$. 
The chromatic index of a graph $G$ is the minimum size of an edge coloring of $G$, and it is denoted by $\chi'(G)$.

The next lemma generalizes in a direction Lemma 4 of
\cite{Bonamy2022}.

\begin{lemma} \label{lemma:extendpacking}
Let~$G$ be a graph, let~$\{X,Y\}$ be a partition of $V(G)$ such that all edges between~$X$ and~$Y$ exist. Then, there exists a packing of size at least~$\min\{1,\frac{|Y|}{\chi'(G[X])}\}|E(G[X])|$.
Moreover, all triangles in such packing have two vertices in $X$ and one vertex in~$Y$.
\end{lemma}
\begin{proof}
Set $G'=G[X]$ and let~$\{M_1,M_2,\ldots,M_{\chi'(G[X])}\}$ be an edge coloring of~$G[X]$.
We can extend each of these matchings to form a packing of~$G$ in the following way.
Let~$\{v_1,v_2,\ldots,v_{\min\{\chi'(G'),|Y|\}}\} \subseteq Y$.
For every~$i \in \{1,2,\ldots,\min\{\chi'(G'),|Y|\}\}$, let~$P_i=\{v_iuw: uw \in M_i\}$.
Observe that every~$P_i$ is a packing in~$G$, as every~$M_i$ is a matching in~$G'$.
Observe also that~$|P_i|=|M_i|$. Also, as every two matchings
in~$\{M_1, M_2, \ldots , M_{\chi'(G')}\}$ are disjoint, and each matching forms triangles with a different vertex from~$Y$,
$P=P_1 \cup P_2 \cup \cdots \cup P_{\min\{\chi'(G'),|Y|\}}$ is a packing in~$G$ of cardinality
${|M_1| + |M_2| + \cdots + |M_{\min\{\chi'(G'),|Y|\}}|}$.
We may assume, without loss of generality,
that~$|M_1| \geq |M_2| \geq  \cdots  \geq  |M_{\chi'(G')}|$.
Hence $|P| \geq \frac{\min\{\chi'(G'),|Y|\}}{\chi'(G')}\sum_{i=1}^{\chi'(G')}|M_i|= \min\{1,\frac{|Y|}{\chi'(G[X])}\}|E(G[X])|$.
\end{proof}

We will also use the next well-known theorem of Vizing.
\begin{proposition}[{\cite{Vizing64}},
see also {\cite[Theorem 5.3.2]{Diestel10}}]\label{prop:tuza-vizing}
For every graph~$G$,~$\chi'(G) \leq \Delta(G) + 1$.
\end{proposition}

For our first two main results, we will use the next lemma. This generalizes an idea that appears implicitly in the proof of Proposition 4 of \cite{Gyori87}, which was later used to show Tuza's conjecture for arbitrary dense graphs \cite{Tuza90}.
\begin{lemma}\label{lemma:main}
   Let~$G$ and~$G'$ be two graphs. Let~$T \subseteq \mathcal{T}(G)$. Let~$P'$ be a packing in~$G'$, and let~$\Pi^{'} \subseteq \mathfrak{S}(G)$. Then,~$\nu(G) \geq \frac{1}{|\Pi^{'}|} \sum_{ t \in T}\sum_{t' \in P'} |\{\pi \in  \Pi^{'}: t = \pi(t^{'})\}|$. 
\end{lemma}
\begin{proof}
Let~$X$ be the random variable defined over~$\Pi^{'}$ by~$X(\pi) = |T \cap \pi(P')|$ for every~${\pi \in \Pi^{'}}$. Note that~$X = \sum_{ t \in T}\sum_{t' \in P'} X_{tt'}$, where~$X_{tt'}$ is the indicator variable of the event~$\{t=\pi(t')\}$. Thus,~$\mathbb{E}[X] = \sum_{ t \in T} \sum_{t' \in P'} \mathbb{P}(t=\pi(t'))$. Note that, since~$T \cap \pi(P')$ is a packing in~$G$ for every~$\pi \in \Pi^{'}$, then~$\nu(G) \geq \mathbb{E}[X]$, which concludes the proof. 
\end{proof}

%% file: split2.tex
The purpose of this section is to show the next theorem for split graphs.

\begin{theorem}\label{teorema:dense-split}
    Let~$G=(K,S,E(G))$ be a split graph on~$n$ vertices. If $\delta(G) \geq \frac{3n}{5}$, then Conjecture \ref{conjecture:tuza} holds.
\end{theorem}

We begin by state two results for the size of a maximum packing when the graph is a complete split graph, and a complete graph.

\begin{proposition}[see also {\cite[Corollary 5]{Bonamy2022}}]\label{prop:prop3}
For every complete split graph~$G=(K,S,E(G))$, there exists a packing in which all triangles have vertices in~$K$ and in~$S$, with size~$\frac{| K | - 1}{2} \cdot \min\{| S |, | K |\}$.
\end{proposition}
\begin{proof}
First suppose that $|S| \geq |K|$.
By Proposition \ref{prop:tuza-vizing}, we have
$\chi'(G[K]) \leq |K|$.
Hence, by Lemma \ref{lemma:extendpacking}, there exists a packing in~$G$ of size at least $|E(G[K])| =\frac{| K | - 1}{2} \cdot \min\{| S |, | K |\}$.
Now suppose that $|S| \leq |K|-1$.
As $\chi'(G[K]) \geq |K|-1$, by Lemma \ref{lemma:extendpacking}, there exists a packing in~$G$ of size at least
$\frac{|S|}{\chi'(G[K])}|E(G[K])|
\geq 
\frac{|S|}{|K|}|E(G[K])|
=\frac{| K | - 1}{2} \cdot \min\{| S |, | K |\}$.
\end{proof}

 \begin{proposition}[{\cite[
Theorem 2]{Feder2012}}]
\label{prop:pK}
For every $n \geq 2$, we have
$\nu(K_n)=\frac{1}{3}({{n}\choose{2}} - k)$, where
$$
k=
\begin{cases}
0 & \text{: }n \text{ MOD } 6 \in \{1,3\}\\
4 & \text{: }n \text{ MOD } 6 =5\\
 \frac{n}{2} & \text{: }n \text{ MOD } 6 \in \{0,2\}\\
\frac{n}{2} + 1 & \text{: } n \text{ MOD } 6 =4.
\end{cases}
$$
Thus,~$\nu(K_n) \geq \frac{1}{3}({{n}\choose{2}} - \frac{n}{2} - \frac{3}{2})$ for every~$n \geq 2$. 
\end{proposition}

We now apply Lemma \ref{lemma:main} and obtain the next two corollaries.

\begin{corollary}[see also {\cite[Proposition 4]{Gyori87}}] \label{cor:nu(G)geqTGnu(K_n)n3}
For any graph~$G$ on~$n$ vertices,~$\nu(G)~\geq~\frac{|\mathcal{T}(G)|\times \nu(K_n)}{\binom{n}{3}}$.
\end{corollary}
\begin{proof}
Let $G' \cong K_n$ with $V(G')=V(G)$, let~$\Pi^{'} = \mathfrak{S}(G')$ and let~$P'$ be a maximum packing in~$G'$. Note that, for a fixed~$(t,t') \in \mathcal{T}(G)  \times P'$, there are exactly~$6(n-3)!$ permutations~$\pi \in \Pi^{'}$ such that~$t = \pi(t')$. Then, by Lemma \ref{lemma:main}, we have~$\nu(G)~\geq~ \frac{|\mathcal{T}(G)|\times |P'| \times 6(n-3)!}{n!}~= ~\frac{|\mathcal{T}(G)|\times \nu(K_n)}{\binom{n}{3}}$.
\end{proof}

\begin{corollary} \label{cor:corollary2}
If~$G=(K,S,E(G))$ is a split graph on~$n$ vertices, then~$\nu(G) \geq \frac{|T'|}{\max\{|S|,|K|\}}$, where~$T'$ is the set of triangles in~$G$ with vertices in~$K$ and~in~$S$. 
\end{corollary}
\begin{proof}
Let~$G' = (K,S,E(G'))$ be a complete split graph and let~$\Pi^{'} \subseteq \mathfrak{S}(G')$ be such that, for each~$\pi \in \Pi^{'}$,~$\pi(v) 
 \in V(K)$ if and only if~$v \in V(K)$.
Note that, by Proposition \ref{prop:prop3}, there exists a packing~$P^{'}$ in~$G'$, in which all triangles have vertices in~$K$ and in~$S$, with size at least~$\frac{|K|-1}{2} \cdot \min\{|S|,|K|\}$. 

Also, for a fixed~$(t,t') \in T' \times P'$, there are exactly~$2(|K|-2)!\times (|S|-1)!$ permutations~${\pi \in \Pi^{'}}$ such that~$t = \pi(t')$. Thus, by Lemma \ref{lemma:main}, we have~$\nu(G) \geq \frac{|T'|\times|P'|\times 2(|K|-2)!\times (|S|-1)!}{|K|! \times |S|!} = \frac{|T'| \times |P'|}{|S| \times \binom{|K|}{2}} \geq \frac{|T'|}{\max\{|S|,|K|\}}$.
 \end{proof}

We also use the next result to bound the size of a triangle hitting for an arbitrary graph.
It is easy to see that the complement of an edge-cut
is a hitting set. Thus, as any graph $G$ has an edge-cut of size
at least $\frac{|E(G)|}{2}+\frac{|V(G)|-1}{4}$
 \cite[Lemma 2] {Bollobas02}, we obtain the next result.

\begin{proposition}\label{prop:tau(G)leqfracE(G)2-frac|V(G)|-14}
For every graph $G$, we have $\tau(G) \leq \frac{|E(G)|}{2}-\frac{|V(G)|-1}{4}$.
\end{proposition}

We now proceed to the proof of Theorem \ref{teorema:dense-split}. Since Conjecture \ref{conjecture:tuza} is valid when~$n \leq$ 8 \cite[Theorem 1.2]{Puleo15}, we will assume that~$n\geq 9$. We begin by proving Conjecture 1 when $|S| \geq |K|$ and $\delta(G) \geq  \frac{1}{2} + \sqrt{\frac{n^2 -2n +2}{8}}$, or $|S| < |K|$ and $\delta(G) \geq  \frac{n+2}{4} + \sqrt{\frac{5n^2 + 2n - 72}{48}} $. This proof will be divided in two cases.

\begin{case}
$|S| \geq |K|$.
\end{case}

We will prove that, if $\delta(G) \geq  \frac{1}{2} + \sqrt{\frac{n^2 -2n +2}{8}} $, then Conjecture \ref{conjecture:tuza} holds. By this condition, we have that $\delta^2(G) - \delta(G) \geq \frac{n(n-2)}{8}$.
Note that~$\nu(G) \geq \sum_{u \in S} \binom{d(u)}{2}/|S|$, by Corollary \ref{cor:corollary2}. Also,~$\tau(G) \leq \binom{|K|}{2}$, because the set of all edges in~$K$ is a hitting set of~$G$. Then, as~$|K| \leq \frac{n}{2}$, \begin{eqnarray} \tau(G) \leq \binom{|K|}{2} \leq \frac{n^2 -2n}{8} \leq  \delta^2(G) - \delta(G). \label{eq:e} \end{eqnarray}
Also, by Cauchy-Schwarz inequality,~$\sum_{u \in S} d^2(u) \geq (\sum_{u \in S} d(u))^2/|S|$. Hence, 
\begin{eqnarray}
2\nu(G)
 &\geq& \frac{1}{|S|}\Big(\sum_{u \in S} d^2(u) - \sum_{u \in S} d(u)\Big) \nonumber \\
 &\geq& \frac{1}{|S|^2}\Big((\sum_{u \in S} d(u))^2 - |S|\sum_{u \in S} d(u)\Big) \nonumber \\
 &\geq&  \delta^2(G) - \delta(G). \label{eq:2}
\end{eqnarray}
By \eqref{eq:e} and \eqref{eq:2}, we have ~$2\nu(G) \geq \tau(G)$, as we want. This finishes the proof of Case 1.

Before continue to the proof of Case 2, we will obtain an important inequality.
Let~${k=|K|}$. By Proposition \ref{prop:pK},~$\nu(K_n) \geq \frac{n^2 -2n -3}{6}$. Thus, as~$n \geq 5$, we have that~$\nu(K_n) \geq \frac{(n-2)(n-1)}{6}$. Also, by Corollary \ref{cor:nu(G)geqTGnu(K_n)n3}, we have that~$\nu(G) \geq \Big(\binom{k}{3} + \sum_{u \in S} \binom{d(u)}{2}\Big) \cdot  \nu(K_n)/\binom{n}{3} \geq \Big(\binom{k}{3} + \sum_{u \in S} \binom{d(u)}{2}\Big)/n$. By Proposition \ref{prop:tau(G)leqfracE(G)2-frac|V(G)|-14}, we have~$\tau(G) \leq \frac{1}{2}(\binom{k}{2} + \sum_{u \in S} d(u) - \frac{n-1}{2})$. Hence, by Cauchy-Schwarz inequality, we have that
\begin{eqnarray}
2\nu(G) - \tau(G) 
&\geq& \frac{2}{n}\binom{k}{3} + \frac{1}{n}\sum_{u \in S} d^2(u) - \frac{1}{n}\sum_{u \in S} d(u) -\frac{1}{2}\binom{k}{2} -\frac{1}{2}\sum_{u \in S} d(u)+\frac{n-1}{2} \nonumber \\
&\geq &
\frac{1}{n}\Big((\sum_{u \in S} d(u))^2/|S|\Big)  - \frac{n+2}{2n}\sum_{u \in S} d(u) 
+\frac{2}{n}\binom{k}{3}-\frac{1}{2}\binom{k}{2} + \frac{n-1}{2} \nonumber \\ 
&\geq&
\frac{n-k}{n}\cdot 
\Big( 
\delta^2(G) - \frac{\delta(G)(n+2)}{2} + \frac{k(k-1)}{n-k} \cdot \frac{4k-3n-8}{12} + \frac{1}{n-k}\binom{n}{2} \Big) 
\label{eq:4}
\end{eqnarray}

We now proceed to the proof of Case 2.
\newpage
\begin{case}
$|S| < |K|$.
\end{case}
We will prove that, if $\delta(G) \geq  \frac{n+2}{4} + \sqrt{\frac{5n^2 + 2n - 72}{48}} $, then Conjecture \ref{conjecture:tuza} holds.\\
By \eqref{eq:4}, we have
\begin{eqnarray}
2\nu(G) - \tau(G)
&\geq&
\frac{n-k}{n} \cdot
\Big( 
\delta^2(G) - \frac{\delta(G)(n+2)}{2} + k \cdot \frac{4k-3n-8}{12} + \frac{n}{2} \Big) \nonumber \\
&\geq&
\frac{n-k}{n}  \cdot \Big(\delta^2(G) - \frac{\delta(G)(n+2)}{2} - (n^2 -5n +6)/24\Big) \nonumber
\\
  &\geq& \frac{n-k}{n} \cdot (-2) \nonumber \\
  &>& -1 \nonumber
\label{eq:5}.
\end{eqnarray}
Where the second inequality is valid
because~$k \cdot \frac{4k-3n-8}{12}$ is an increasing function on $k$ and~$k \geq \frac{n+1}{2}$; and the third inequality is valid by the case condition.
Since~$2\nu(G) - \tau(G)$ is integer, it implies that~$2\nu(G) - \tau(G) \geq 0$. This finishes the proof of Case 2.

Note that $ \lceil \frac{3n}{5} \rceil \geq \max \{ \lceil \frac{1}{2} + \sqrt{\frac{n^2 -2n +2}{8}} \rceil ,\lceil \frac{n+2}{4} + \sqrt{\frac{5n^2 + 2n - 72}{48}}\rceil \}$
for $n=9$ and $n \geq 11$. For $n=10$, $\max \{ \lceil \frac{1}{2} + \sqrt{\frac{n^2 -2n +2}{8}} \rceil ,\lceil \frac{n+2}{4} + \sqrt{\frac{5n^2 + 2n - 72}{48}}\rceil \} = 7$ and $\frac{3n}{5} = 6$. Thus, the conjecture holds if $\delta(G) \geq 7$. Now, if $\delta(G) = 6$, then $k \geq 6$.  By \eqref{eq:4}, we have that \begin{eqnarray}
2\nu(G) - \tau(G)
&\geq&
\frac{n-k}{n}
\Big( 
\delta^2(G) - \delta(G)(n+2)/2 + \frac{k(k-1)}{n-k} \cdot \frac{4k-3n-8}{12} + \frac{1}{n-k}\binom{n}{2} \Big) \nonumber \\
&\geq& \frac{(n-k)}{n} \cdot \frac{5}{2} \nonumber \\
&\geq& 0. \nonumber
\end{eqnarray}
This finishes the proof of Theorem \ref{teorema:dense-split}.

%% file: 3partite.tex
Szestopalow {\cite[Theorem 4.1.5]{Sze16}} showed that~$\tau(G)/\nu(G) \leq \frac{28}{15} \approx 1.87$ for every tripartite graph~$G$.
In this section, we will improve this bound when
$G$ is dense (Theorem \ref{thm:3partite}). We will use the next property, also known as König's line coloring Theorem.

\begin{proposition}[{\cite{Konig1916}}] \label{prop:Konig}
For every bipartite graph~$G$,~$\chi'(G) \leq \Delta(G)$.
\end{proposition}

We now proceed to the proof of our main theorem.

\begin{theorem}\label{thm:3partite}
For every tripartite graph~$G=(I_1,I_2,I_3,E(G))$ with~$n$ vertices and~$m > \frac{n^2}{4}$ edges, we have~$\tau(G) \leq \frac{n^2}{3(4m-n^2)}\times\nu(G)$.
\end{theorem}
\begin{proof}

Let~$G' = (I_1,I_2,I_3,E(G'))$ be a complete tripartite graph
with~$|I_1| \geq |I_2| \geq |I_3|$.
 Let~$\Pi^{'}$ be the maximal subset of $\mathfrak{S}(G^{'})$ such that, for each~$\pi \in \Pi^{'}$,~$\pi(v) 
 \in I_i$ if and only if~$v \in I_i$, for each~$1\leq i \leq 3$. Let~$P'$ be a maximum packing in~$G'$.
Let us suppose without loss of generality that~$|I_1| \geq |I_2| \geq |I_3|$.
As $G[I_2 \cup I_3]$ is bipartite, we have
$\chi'(G[I_2 \cup I_3])=|I_2|$ by Proposition \ref{prop:Konig}.
Thus, by Lemma \ref{lemma:extendpacking}, with $X=I_2 \cup I_3$ and $Y=I_1$, we have
$|P'| \geq |I_2||I_3|$. 
 
Note that, for a fixed~$(t,t') \in \mathcal{T}(G) \times P'$, there are exactly~$(|I_1|-1)!\cdot (|I_2|-1)! \cdot (|I_3|-1)!$ permutations~$\pi \in \Pi^{'}$ such that~$t =  \pi(t')$. Thus, by Lemma \ref{lemma:main}, we have~$\nu(G) \geq 
\frac{|\mathcal{T}(G)| \cdot |P'|}{|I_1|\cdot|I_2|\cdot|I_3|} \geq \frac{|\mathcal{T}(G)|}{|I_1|}$.
By Bollobás {\cite[Corollary 6.1.9]{Bollobas78}}(see also {\cite[Figure 1]{Fisher89}}),~$|\mathcal{T}(G)| \geq \frac{n}{9} \cdot (4m - n^2)$. So,~$\nu(G) \geq \frac{|\mathcal{T}(G)|}{|I_1|} \geq \frac{(4m − n^2)n}{9|I_1|}$. Now, as~$\tau(G) \leq |I_2|\cdot |I_3|$, we have that
\begin{eqnarray}
\frac{\tau(G)}{\nu(G)} \leq  \frac{9 |I_1|\cdot |I_2|\cdot |I_3|}
{(4m-n^2)n} \leq
\frac{\frac{n^3}
{3}}
{(4m-n^2)n} = \frac{1}
{3} \cdot \frac{n^2}
{(4m-n^2)} \nonumber,
\end{eqnarray}
and the proof follows.\\
\end{proof}
As stated in the introduction of this section, given a tripartite graph~$G$, the best known upper bound for~$\frac{\tau(G)}{\nu(G)}$ is~$\frac{28}{15} \approx 1.87$. The following corollary shows that this bound is improved if~$G$ is dense enough.
\begin{corollary}
For any~$\alpha>0$, every tripartite graph~$G$ with more than~$(\frac{1 + 3\alpha}{12\alpha})n^2$ edges satisfies~$\tau(G) < \alpha \nu(G)$. In particular, if~$G$ has more than~$\frac{33n^2}{112}$ edges then~$\tau(G) < \frac{28}{15}\nu(G)$.
\end{corollary}

%% file: 4partite.tex
\section{Complete 4-Partite Graphs}
\label{section:tuza-complete-4partite}

Aparna~\etal~showed that Tuza's Conjecture holds for
4-partite graphs~\cite[Corollary 7]{Aparna11}.
In this section, we improve this result for complete 4-partite graphs (Theorem \ref{theorem:tuza-4-partite-complete}).
We begin by stating an auxiliary result, known as the Ore–Ryser Theorem.
Given a graph~$G$ and a function~$f:V(G) \rightarrow \mathbb{Z}^{+}$, an~$f$-factor
is a spanning subgraph~$H$ of~$G$ such that~$d_H (v) = f (v)$ for every~$v \in V (G)$.

\begin{proposition}\cite[Theorem 1]{Vandenbussche2013}
\label{prop:Vandenbussche}
Let~$G=(C,D,E(G))$ be a bipartite graph.
Let $f:V(G) \rightarrow \mathbb{Z}^{+}$.
$G$ has an
$f$-factor if and only if $f(C) = f(D)$ and,
for all 
$D' \subseteq D$,
$$
f(D') \leq \sum_{y \in N(D')} \min\{f(y), |N(y)\cap D'|\}.
$$
\end{proposition}

We now prove the main theorem of this section.

\begin{theorem}\label{theorem:tuza-4-partite-complete}
For every complete 4-partite graph~$G$ on at least five vertices,~${\tau(G) \leq \frac{3}{2}\nu(G)}$.
Moreover, this bound it tight.
\end{theorem}
\begin{proof}
Let~$G=(A,B,C,D,E(G))$, with~$a:=|A|,b:=|B|, c:=|C|, d:=|D|$
and~$a  \geq b \geq c \geq d$.
As~$|V(G)| \geq 5$, we have~$a\geq 2$.
For every~$P$,~$Q \in \{A,B,C,D\}$, we say that 
the~$PQ$ \textit{edges} are the edges of~$G$ with one end
in~$P$ and the other in~$Q$.

Suppose for a moment that~$a\geq b+c+1$.
Let~$G'=G[B \cup C \cup D]$.
Note that, by Proposition~\ref{prop:tuza-vizing},~$\chi'(G') \leq b+c+1 \leq a$.
Thus, by Lemma \ref{lemma:extendpacking},~$\nu(G)\geq |E(G')| = bc+cd+bd$.
Also, the set of~$BC$ edges joined
to the set of~$CD$ and~$BD$ edges form a hitting set of~$G$, with cardinality~$bc+cd+bd$.
Thus~$\tau(G)\leq  bc+cd+bd  \leq \nu(G)$
and the proof follows. Hence,
from now on,
we may assume that~$a\leq b+c$.
We divide the rest of the proof on whether~$a>c+d$ or not.
For the rest of the proof, we observe that the set of~$BC$ edges joined to the set of~$AD$ form a hitting set. Hence,
\begin{equation*}\label{eq:tauleqad+bc}
\tau(G)\leq ad+bc.
\end{equation*}

\setcounter{case}{0}
\begin{case}~$a>c+d$.

First suppose that~$a\geq b+1$.
We define a function $f:V(G[C \cup D]) \rightarrow \mathbb{Z}^{+}$ as follows.
For every $v \in D$, we set $f(v)=a-b-1$,
and for every $v \in C$, we set
$f(v)$ to either $\lceil \frac{d}{c}(a-b-1) \rceil$ or $\lfloor \frac{d}{c}(a-b-1) \rfloor$
such that $f(C)=f(D)$.
Let $D' \subseteq D$. Note that
$f(D')=(a-b-1)|D'|$.
If $|D'| \leq \lfloor \frac{d}{c}(a-b-1) \rfloor$,
then
$\sum_{y \in N(D')} \min\{f(y), |N(y) \cap D'|\}=
\sum_{y \in N(D')} |D'|=|C||D'| \geq (a-b-1)|D'|=f(D')$.
Otherwise,
$\sum_{y \in N(D')} \min\{f(y), |N(y) \cap D'|\}=
\sum_{y \in N(D')} f(y)=f(C)=f(D) \geq f(D')$.
Hence, by Proposition \ref{prop:Vandenbussche},
there exists an $f$-factor of $G[C \cup D]$, say
$G'_{CD}$.

Let~$G'=G[B \cup D] \cup G[B \cup C] \cup G'_{CD}$.
Observe that~$\Delta(G') \leq a-1$.
Indeed, if~$v\in B$, then~$d_{G'}(v)=c+d \leq a - 1$;
if~$v\in C \cup D$, then~$d_{G'}(v) \leq b+(a-b-1) = a - 1$.
Hence, as~$\chi'(G')\leq \Delta(G')+1 \leq a$
by Proposition~\ref{prop:Konig}, we have that~$\nu(G) \geq |E(G')| = bc+bd+(a-b-1)d=ad+bc-d$ by Lemma \ref{lemma:extendpacking}.
As~$a\geq 2$, we have that~$3d \leq ad+bc$.
Hence,
$$
\tau(G) \leq ad+bc
        \leq \frac{ad+bc}{ad+bc-d} \cdot \nu(G)
        \leq \frac{3}{2} \nu(G)
.$$

Now, let us suppose that~$a \leq b+1$.
Let~$G'=G[B \cup D] \cup G[B \cup C]$
and note that, as $G'$ is bipartite,~$\chi'(G')\leq \max\{b,c+d\} \leq a$ by Proposition \ref{prop:Konig}.
Thus, by Lemma~\ref{lemma:extendpacking},
$\nu(G) \geq |E(G')| = bc+bd$, and
$$
\tau(G) \leq ad+bc
        \leq \frac{ad+bc}{bc+bd} \cdot \nu(G)
        \leq \frac{bd+d+bc}{bc+bd} \cdot \nu(G)
         \leq \frac{3}{2} \nu(G)
.$$
This finishes the proof of Case 1.
\end{case}
\input{figura.tex}
\newpage
\begin{case}~$a \leq c+d$.

For this case, we will use the next two claims.
Their proofs rely heavily on Lemma \ref{lemma:extendpacking}.

\begin{claim}\label{claim:nu(G)geqfracab+c+1bc+bd+cd}
$\nu(G) \geq \frac{a}{b+c+1}(bc+bd+cd)$.
\end{claim}
\begin{proof}
Consider the tripartite graph~$G'=G[B \cup C \cup D]$, we have that~$\chi'(G) \leq b+c+1$ by Proposition \ref{prop:tuza-vizing}.
Thus, by Lemma \ref{lemma:extendpacking}, with
$X=B \cup C \cup D$ and $Y=A$,
$\nu(G) \geq \frac{a}{b+c+1}(bc+bd+cd)$.
\end{proof}

\begin{claim}\label{claim:nu(G)geqab+fracc+d-a4}
$\nu(G) \geq ab + \frac{(c+d-a)^2-1}{4}$.
\end{claim}
\begin{proof}
Let~$x=\lfloor (c+d-a)/2 \rfloor$.
Set $C=\{u_1,u_2, \ldots , u_c\}$,
and let
$C'= \{u_1,u_2, \ldots u_x\}$
and $C''= C \setminus C'$.
Set $D=\{v_1,v_2, \ldots v_d\}$, and let
$D'= \{v_1,v_2, \ldots v_{a-c+x}\}$
and $D''= D \setminus D'$.
By Proposition \ref{prop:Konig}, $\chi'(G[A\cup B]) \leq a$.
Thus, by Lemma \ref{lemma:extendpacking},
with $X=A\cup B$ and $Y=C'' \cup D'$, there exists a packing $P$ in $G$ with size at least $ab$.

Let $G'$ be the subgraph of $G$ resulting by removing all edges between $A$ and $B$.
Consider now the bipartite graph~$G'[C'\cup A \cup B]$ and note that $\chi'(G'[C'\cup A \cup B]) \leq a+b$ by Proposition \ref{prop:Konig}.
Thus, by Lemma \ref{lemma:extendpacking},
with $G=G'$, $X=C'\cup A \cup B$
and $Y=D''$, there exists a packing $P'$ in $G'$ with size at least $\frac{|D''|}{a+b}|E(G'[C'\cup A \cup B])|=x(c+d-a-x)$ (Figure \ref{fig:claimcase2}).
As $P$ and $P'$ are disjoint to each other,
$\nu(G) \geq ab + x (c+d-a-x) = ab + \lfloor \frac{c+d-a}{2}\rfloor \lceil \frac{c+d-a}{2} \rceil \geq ab + \frac{(c+d-a)^2-1}{4}$.
\end{proof}
We now continue with the proof of Case 2. By Claim \ref{claim:nu(G)geqab+fracc+d-a4}, if~$d\leq b/2$, then~$\tau(G) \leq ad+bc \leq \frac{ab}{2} + ab \leq \frac{3}{2}\nu(G)$ and we are done.
So, from now on, we may assume that 
$b\leq 2d-1$, or equivalently, $d\geq \frac{b+1}{2}$.
Suppose for a moment that~$a \geq b+1$. Then, as~$d\geq \frac{b+1}{2}$, by Claim \ref{claim:nu(G)geqfracab+c+1bc+bd+cd}, \begin{eqnarray*}
3\nu(G)/2 - \tau(G)
&\geq&
\frac{3a(bc+bd+cd)}{2(b+c+1)}-ad-bc \\
&=&
a\cdot \frac{3bc+d(b+c-2)}{2(b+c+1)}-bc \\
&\geq&
(b+1)\cdot \frac{3bc+\frac{b+1}{2}\cdot (b+c-2)}{2(b+c+1)}-bc \\
&=&
\frac{6(b+1)bc+(b+1)^2(b+c-2)-4bc(b+c+1)}{4(b+c+1)} \\
&=&
\frac{(b+4c)(b^2-bc)+3b(c-1)+(b+1)c-2}{4(b+c+1)} \\
&\geq&
0,
\end{eqnarray*}
and the proof follows. Hence, from now on, we may assume that~$a=b$. Thus, by Claim~\ref{claim:nu(G)geqab+fracc+d-a4},
\begin{eqnarray*}
3\nu(G)/2- \tau(G)&\geq&
\frac{3}{2}(ab + \frac{(c+d-a)^2-1}{4})-ad-bc \\
&=&
\frac{3}{2}(b^2 + \frac{(c+d-b)^2-1}{4})-b(c+d) \\
&=&
\frac{15b^2 +3(c+d)^2-14b(c+d)-3}{8}\\
&=&
\frac{1}{8}
(5b-3(c+d))(3b-(c+d))-\frac{3}{8} \\
&\geq&
\frac{b}{8}
(5b-3(c+d))-\frac{3}{8}.
\end{eqnarray*}
Hence, if~$5b\geq 3c+3d$ we are done.
So, from now on, me may assume that
$5b < 3c+3d$, or equivalently, that $c+d > 5b/3$.
Note also that, as~$5b<3b+3d$, we have~$d >2b/3$.
Now, by Claim~\ref{claim:nu(G)geqfracab+c+1bc+bd+cd}, we have
\begin{eqnarray*}
3\nu(G)/2 - \tau(G) +1 
&\geq&
\frac{3a(bc+bd+cd)}{2(b+c+1)}-ad-bc+1 \\
&=&
\frac{3b(b(c+d)+cd)}{2(b+c+1)}-b(c+d) + 1\\
&=&
\frac{b(c+d)(b-2c-2)+3bcd+2(b+c+1)}{2(b+c+1)}\\
&>&
\frac{\frac{5b^2}{3}(b-2c-2)+3bc \cdot \frac{2b}{3}+2(b+c+1)}{2(b+c+1)}\\
&=&
\frac{b^2(5b-10)+6(b+1)-c(4b^2-6)}{6(b+c+1)}.
\end{eqnarray*}

Suppose for a moment that~$a=b\geq c+1$.
Then,
\begin{eqnarray*}
3\nu(G)/2 - \tau(G) +1
&\geq&
\frac{b^2(5b-10)+6(b+1)-(b-1)(4b^2-6)}{6(b+c+1)}\\
&=&
\frac{b^3-6b^2+12b}{6(b+c+1)} \\
&>&
0,
\end{eqnarray*}
and the proof follows. Hence, we may assume
that~$a=b=c$. 
If $a=b=c=d$, then, by Claim \ref{claim:nu(G)geqfracab+c+1bc+bd+cd},
$3\nu(G)/2 - \tau(G) \geq
 \frac{3}{2} \lceil \frac{3a^3}{2a+1}\rceil - 2a^2
 \geq -\frac{1}{2}
$
and we are done.

We finally, consider the case when $a=b=c$ and $a\neq d$.
For this, we need to strengthen Claim \ref{claim:nu(G)geqfracab+c+1bc+bd+cd} with the help of the next result.

\begin{proposition}\cite[Theorem 4]{Akbari2012}
\label{prop:chi(G)=Delta(G)}
Let $G$ be a graph. If all vertices of maximum degree induce a forest, then $\chi'(G)= \Delta(G)$.
\end{proposition}

\begin{claim}\label{claim:nu(G)geqfracab+cbc+bd+cd}
If $c \neq d$, then $\nu(G) \geq \frac{a}{b+c}(bc+bd+cd)$
\end{claim}
\begin{proof}
Consider the tripartite graph~$G'=G[B \cup C \cup D]$.
As all vertices in $D$ has degree $b+c$,
all vertices in $C$ has degree $b+d$,
and $c\neq d$, we have that all vertices of maximum degree in $G'$ are exactly the vertices in $D$, which induce a forest. 
Thus, by Proposition~\ref{prop:chi(G)=Delta(G)},
we have that~$\chi'(G) = b+c$.
Hence, by Lemma \ref{lemma:extendpacking}, with
$X=B \cup C \cup D$ and $Y=A$,
$\nu(G) \geq \frac{a}{b+c}(bc+bd+cd)$.
\end{proof}

Recall that $a=b=c \neq d$, and $d>2a/3$.
By Claim \ref{claim:nu(G)geqfracab+cbc+bd+cd},
\begin{eqnarray*}
3\nu(G)/2 - \tau(G)
&\geq&
\frac{3}{4}(a^2+2ad) -a^2-ad\\
&=&
\frac{1}{4}(2ad-a^2)\\
&>&
0.
\end{eqnarray*}
This concludes the proof of Case 2.
\end{case}

Finally, note that if $G$ is the complete 4-partite graph on 5 vertices, then $\tau(G)= \frac{3\nu(G)}{2}$.
This concludes the proof of the theorem.
\end{proof}

%% file: figura.tex
\begin{figure}[h]
    \centering
    \begin{tikzpicture}[scale=.72]
        \def\radius{1.5cm}
        \foreach \i/\label in {
    1/{$B$},
    2/{$A$},
    3/{$C^{''}$},
    4/{$C^{'}$},
    5/{$D^{''}$},
    6/{$D^{'}$}
} {
            \ifnum\i<3
                \coordinate (v\i) at ({\i*60}:3.5); 
            \draw (v\i) circle (\radius);
            \node at ({\i*60}:5.5) {\label};
                \foreach \j in {1,2,3,4} {
                    \pgfmathsetmacro{\y}{(\j+1)*0.4} 
                    \coordinate (p\i\j) at ($(v\i) + (90:\radius) + (0,-\y)$);
                    \fill (p\i\j) circle (2pt); 
                }
            \else
                \ifnum\i=3
                  \coordinate (v\i) at ({\i*58}:4); 
            \draw (v\i) circle (\radius*0.8);
            \node at ({\i*58}:5.7) {\label};
                  \foreach \j in {1,2} {
                    \pgfmathsetmacro{\y}{(\j+1)*0.6} 
                    \coordinate (p\i\j) at ($(v\i) + (90:\radius) + (0,-\y)$);
                    \node[anchor=east] at (p\i\j) {$u_{\pgfmathparse{int(\j+2)}\pgfmathresult}$};
                    \fill (p\i\j) circle (2pt); 
                 }
                \fi
                \ifnum\i=6
                  \coordinate (v\i) at ({\i*61}:4); 
            \draw (v\i) circle (\radius*0.8);
            \node at ({\i*61}:5.7) {\label};
                  \foreach \j in {1,2} {
                    \pgfmathsetmacro{\y}{(\j +1)*0.6} 
                    \coordinate (p\i\j) at ($(v\i) + (90:\radius) + (0,-\y)$);
                    \node[anchor=west] at (p\i\j) {$v_{\pgfmathparse{int(\j+1)}\pgfmathresult}$};
                    \fill (p\i\j) circle (2pt); 
                 }
                \fi
                \ifnum\i=4
                  \coordinate (v\i) at ({\i*56}:2.8); 
            \draw (v\i) circle (\radius*0.8);
            \node at ({\i*56}:4.5) {\label};
                  \foreach \j in {1,2} {
                    \pgfmathsetmacro{\y}{(\j+1)*0.6} 
                    \coordinate (p\i\j) at ($(v\i) + (90:\radius) + (0,-\y)$);
                    \node[anchor=east] at (p\i\j) {$u_{\pgfmathparse{int(\j)}\pgfmathresult}$};
                    \fill (p\i\j) circle (2pt); 
                 }
                \fi
                \fi
                \ifnum\i=5
                  \coordinate (v\i) at ({\i*64}:3); 
            \draw (v\i) circle (\radius*0.6);
            \node at ({\i*64}:4.5) {\label};
                  \foreach \j in {1} {
                    \pgfmathsetmacro{\y}{(\j+3)*0.4} 
                    \coordinate (p\i\j) at ($(v\i) + (90:\radius) + (0,-\y)$);
                    \node[anchor=west] at (p\i\j) {$v_{1}$};
                    \fill (p\i\j) circle (2pt); 
                 }
                \fi
         
        }
        \foreach \i in {1,2,3,4} {
            \foreach \j in {1,2,3,4} {
                \draw (p1\i) -- (p2\j);
            }
            \draw (p31) -- (p2\i);
            \draw (p32) -- (p2\i);
            \draw (p31) -- (p1\i);
            \draw (p32) -- (p1\i);
            \draw (p61) -- (p2\i);
            \draw (p62) -- (p2\i);
            \draw (p61) -- (p1\i);
            \draw (p62) -- (p1\i);

            \draw[dashed] (p42) -- (p51);
            \foreach \i in {4} {
             \draw[dashed] (p1\i) -- (p51);
            }
            \foreach \i in {4} {
             \draw[dashed] (p1\i) -- (p42);
            }
            
        }
        \draw[dashed] (p41) -- (p51);
        \draw[dashed] (p24) -- (p51);
        \draw[dashed] (p24) -- (p41);
    \end{tikzpicture}
    \caption{A complete 4-partite graph with \( a = 4 \), \( b = 4 \), \( c = 4 \), \( d = 3 \), and \( x = 2 \). Packing \( P \) is formed by the solid edges, and packing \( P' \) is formed by the dashed edges. 
}\label{fig:claimcase2}
\end{figure}
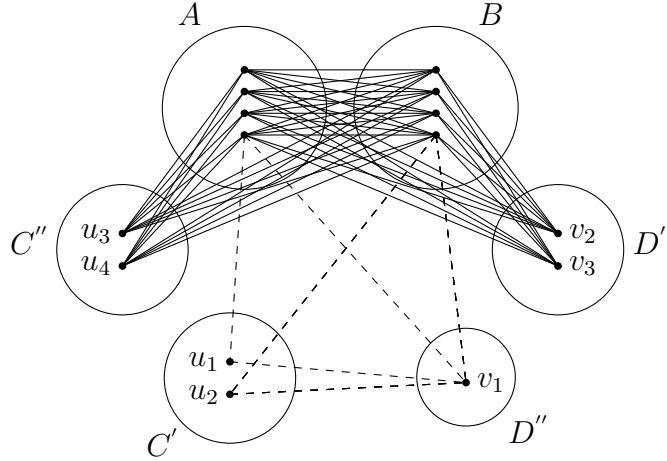

%% file: remarks.tex
Although the directed version of Tuza's conjecture has been already solved by McDonald \etal~\cite{McDonald2020},
the undirected version remains hard to prove even for the case of split graphs.
In this paper, we progress towards this goal by showing
that Tuza's conjecture is valid for dense split graphs.
Our main technique use a probabilistic argument and we also use it to show a result for dense tripartite graphs.

For complete $4$-partite graphs, we obtain an improved result: we show that $\tau(G) \leq \frac{3}{2}\nu(G)$ for every complete 4-partite graph $G$ on at least 5 vertices. We note that this tight bound also exists for planar triangulantions \cite{Botler2020}, $K_4$-free planar graphs \cite{Haxell12}, and others subclasses of $K_4$-free graphs \cite{Munaro}~. An interesting question will be characterize the graphs in which this factor is attained tightly. Also, it will be interesting to find new classes of graphs that satisfy this upper bound.

We believe the techniques showed here can be extended to show new results for dense~$k$-partite graphs and other graph classes, as chordal graphs.